\documentclass[draft]{amsart}
\usepackage{amssymb,latexsym,amscd, amsthm, epsfig,amsmath,amssymb}
\usepackage{color}

 \newtheorem{thm}{Theorem}[section]
 
 \newtheorem{lem}[thm]{Lemma}
 \newtheorem{prop}[thm]{Proposition}
 \theoremstyle{definition}
 
 \theoremstyle{remark}
 \newtheorem{rem}[thm]{Remark}
 \theoremstyle{definition}

\def\move-in{\parshape=1.75true in 5true in}

\usepackage{xypic}

\begin{document}

\title{{
The {Waring} rank of the sum of pairwise coprime monomials}}

\author[E. Carlini]{Enrico Carlini}
\address[E. Carlini]{Dipartimento di Scienze Matematiche, Politecnico di Torino, Turin, Italy}
\email{enrico.carlini@polito.it}

\author[M.V.Catalisano]{Maria Virginia Catalisano}
\address[M.V.Catalisano]{Dipartimento di Ingegneria Meccanica, Energetica, della Pro\-du\-zio\-ne, dei Trasporti e dei Modelli Matematici, Universit\`{a} di Genova, Genoa, Italy.}
\email{catalisano@diptem.unige.it}

\author[A.V. Geramita]{Anthony V. Geramita}
\address[A.V. Geramita]{Department of Mathematics and Statistics, Queen's University, King\-ston, Ontario, Canada and Dipartimento di Matematica, Universit\`{a} di Genova, Genoa, Italy}
\email{Anthony.Geramita@gmail.com \\ geramita@dima.unige.it  }

\maketitle


\begin{abstract}
{ In this paper we compute the Waring rank of any polynomial of
the form $F=\sum_{i=1}^r M_i$,  where the $M_i$  are pairwise
coprime monomials, i.e., $GCD(M_i,M_j)=1$ for $i\neq j$. In
particular, we determine the Waring rank of any monomial. As an
application we show that certain monomials in three variables give examples
of forms of rank higher than the generic form. As a further
application we produce a sum of power decomposition for any form
which is the sum of pairwise coprime monomials.}
\end{abstract}


\section{Introduction}

Let $k$ be an algebraically closed field of characteristic zero
and $k[x_1,\ldots,x_n]$ be the standard graded polynomial ring in
$n$ variables. Given a degree $d$ form $F$ the {\it Waring Problem
for Polynomials} asks for the least value of $s$ for which there
exist linear forms $L_1,\ldots,L_s$ such that
\[F=\sum_1^s L_i^d \ . \]
This value of $s$ is called the {\it Waring rank} of $F$ (or
simply the {\it rank} of $F$) and will be denoted by
$\mathrm{rk}(F)$.

There was a long-standing conjecture describing the rank of a
generic form $F$ of degree $d$, but the verification of that
conjecture was only found relatively recently in the famous work
of J. Alexander and A. Hirschowitz \cite{AH95}.  However, for a
given {\it specific} form $F$ of degree $d$ the value of
$\mathrm{rk}(F)$ is not known in general. Moreover, in the general
situation, there is no effective algorithmic way to compute the
rank of a given form. Algorithms exist in special cases, e.g. when
$n=2$ for any $d$ (i.e. the classical algorithm attributed to
Sylvester) and for $d=2$ any $n$ (i.e. finding the canonical forms for
quadratic forms).

Given
this state of affairs, several attempts have been made to compute
the rank of specific forms.  One particular family of examples
that has attracted attention is the collection of monomials.

A few cases where the ranks of specific monomials are computed can
be found in \cite{LM} and in \cite{LandsbergTeitler2010}. In
\cite{RS2011} the authors determine $\mathrm{rk}(M)$ for the
monomials
\[M=\left(x_1\cdot\ldots\cdot x_n\right)^m\] for any $n$ and $m$. In particular, they show that
$\mathrm{rk}(M)=(m+1)^{n-1}$.  In this paper we completely solve
the Waring Problem for monomials in  Proposition
\ref{monmialsumofpoercorol}  showing that
\begin{eqnarray}  \label{monomio} \ \ \ \
\mathrm{rk}({x_1}^{a_1}\cdot\ldots\cdot {x_n}^{a_n})= \frac
1{a_1+1} \Pi_{i=1}^n (a_i+1)= \left\{
\begin{matrix}
1 & \hbox{ for } n=1  \cr \\
\Pi_{i=2}^n (a_i+1) &  \hbox{ for }  n\geq 2\cr
\end{matrix} \right. ,
\end{eqnarray}
where $1\leq a_1\leq\ldots\leq a_n$.
A lengthier proof of this result was first obtained in
\cite{CCG11} and then, in a different form, in
\cite{BBT2012}.

Our approach to solving the Waring Problem for specific
polynomials follows a well known path, namely the use of the
Apolarity Lemma \ref{apolarityLEMMA} to relate the computation of
$\mathrm{rk}(F)$ to the study of ideals of reduced points
contained in the ideal $F^\perp$. Using these ideas we obtained a
complete solution to the Waring problem for polynomials which are
the sum of coprime monomials, see Theorem \ref{mainthm}. More
precisely, if $F=M_1+\ldots +M_r$ where the monomials $M_i$ are
such that $G.C.D.(M_i,M_j)=1,i\neq j$ and $\mathrm{deg}(F)>1$,
then
\[\mathrm{rk}(F)=\mathrm{rk}(M_1)+\ldots +\mathrm{rk}(M_r).\]

Using our knowledge of the rank  we obtained two interesting
applications. We showed that, only in three variables and for
degree high enough, certain monomials provide examples of forms
having rank higher than the generic form, see Proposition
\ref{nleq3monomial}. Finally, we find a minimal sum of powers
decomposition for forms which are the sum of pairwise coprime
monomials. In the case of monomials this result appeared in
\cite{CCG11} and was then improved in \cite{BBT2012}.

The main results of this paper were obtained in July 2011 when the
authors were visiting the University of Coimbra in Portugal. The
authors wish to thank GNSAGA of INDAM for the financial support
during their visit.

\section{Basic facts}

We consider $k$, an algebraically closed field of characteristic
zero and the polynomial rings
{
\[S=k[x_{1,1} ,\ldots, x_{1,n_1},
\ldots \ldots
,x_{r,1}   ,\ldots, x_{r,n_r} ],\]
\[T=k[X_{1,1} ,\ldots, X_{1,n_1},
\ldots \ldots
,X_{r,1}   ,\ldots, X_{r,n_r} ].\]
}
We make $T$ act via differentiation on $S$, e.g. we think of $X_{i,j}
=\partial/ \partial x_{i,j}$. (see, for example, \cite{Ge} or \cite{IaKa}
). We refer to a polynomial in $T$ as $\partial$, instead
of using capital letters. In particular, for any form $F$ in $S_d$
we define the ideal $F^\perp\subseteq T$ as follows:
\[F^\perp=\left\{\partial\in T : \partial F=0\right\}.\]

Given a homogeneous ideal $I\subseteq T$ we denote by $$HF(T/I,i)=
\dim_kT_i -\dim_k I_i$$ its {\it Hilbert function} in degree $i$.
It is well known that for all $i > > 0$ the function $HF(T/I,i)$
is a polynomial function with rational coefficients, called the
{\it Hilbert polynomial} of $T/I$.  We say that an ideal $I\subseteq
T$ is {\it one dimensional} if the Krull dimension of $T/I$ is one,
equivalently the Hilbert polynomial of $T/I$ is some integer
constant, say $s$.  The integer $s$ is then called the {\it
multiplicity} of $T/I$.  If, in addition, $I$ is a radical ideal,
then $I$ is the ideal of a set of $s$ distinct points.  We will
use the fact that if $I$ is a one dimensional saturated ideal of
multiplicity $s$, then $HF(T/I, i)$ is always $\leq s$.

Our main tool is the {\it Apolarity Lemma}, the proof of which can
be found in \cite[Lemma 1.31]{IaKa}.

\begin{lem}\label{apolarityLEMMA}
A homogeneous degree $d$ form $F\in S$ can be written as
\[F=\sum_{i=1}^s L_i^d ,\ \  \ L_i \hbox{ pairwise linearly independent}\]
 if and only if there exists $I\subseteq F^\perp$ such that $I$ is the ideal of a set of $s$ distinct points.
\end{lem}

We conclude with the following trivial, but useful, remark showing
that the rank of a form does not vary by adding variables to the
polynomial ring.

\begin{rem}\label{leastvarREM}{
The computation of the rank of  $F$ is independent of the polynomial ring in which we consider $F$.    To see this,
consider a rank $d$ form $F\in k[x_1,\ldots,x_n]$ and suppose we
know $\mbox{rk}(F)$. We can also consider $F\in
k[x_1,\ldots,x_n,y]$ and we can look for a sum of powers
decomposition of $F$ in this extended ring. If
\[F(x_1,\ldots,x_n)=\sum_1^r
\left(L_i(x_1,\ldots,x_n,y)\right)^d,\] then, by setting $y=0$, we
readily get $r\geq \mbox{rk}(F)$. Thus, by adding variables we can
not get a sum of powers decomposition involving fewer summands.
Moreover, if $r$ is the minimal length of a sum of powers
decomposition of $F$ in the extended ring, we readily get $r=
\mbox{rk}(F)$. In particular, given a monomial
\[M=x_1^{a_1}\cdot\ldots\cdot
x_n^{a_n},\] with $1\leq a_1\leq\ldots\leq a_n$ it is enough to
work in $k[x_1,\ldots,x_n]$ in order to compute
$\mbox{rk}(M)$.}\end{rem}

\section{Main result}

It is useful to recall the following. Let $I\subseteq {T}$ be an ideal
and $\partial\in T_1$ a linear  homogeneous differentiation. If $\partial$ is
not a zero divisor in $T/I$ then
\begin{eqnarray}  \label{HF}
HF(T/I,t)=\sum_{i=0}^t
HF(T/(I+(\partial)),i)).
\end{eqnarray}

We first compute the rank of any monomial. Thus, we only consider
the case $r=1$ and, just for this result, we drop the double index
notation, i.e. we abuse notation and we let $S=k[x_1,\ldots,x_n]$
and $T=k[X_1,\ldots,X_n]$.

\begin{prop}\label{monmialsumofpoercorol}
Let $n\geq 1$ and $1\leq a_1\leq \ldots \leq a_n$. If
\[M=x_1^{a_1}\cdot\ldots\cdot x_n^{a_n},\]
then $\mbox{rk}(M)=\frac 1{a_1+1} \Pi_{i=1}^n (a_i+1)$.
\end{prop}
\begin{proof}
If $n=1$, then $M$ is the power of a variable and
$\mbox{rk}(M)=1$; we can then assume $n>1$. The perp ideal of $M$
is $M^\perp=(X_1^{a_1+1},\ldots,X_n^{a_n+1})$ and hence
\[I=
(X_n^{a_n+1}-X_1^{a_n+1},\ldots,X_2^{a_2+1}-X_1^{a_2+1})\subseteq
M^\perp .\] As $I$ is the ideal of a complete intersection scheme
of $\frac 1{a_1+1} \Pi_{i=1}^n (a_i+1)$ distinct points, the
Apolarity Lemma yields
\[\mbox{rk}(M)\leq \frac{1}{a_1+1}
\Pi_{i=1}^n (a_i+1).\]

We now consider $I\subseteq M^\perp$ the ideal of a scheme of $s$
distinct points; to complete the proof it is enough to show that
$s\geq \frac{1}{a_1+1} \Pi_{i=1}^n (a_i+1)$. To do this, we set
 $I'=I:(X_1)$ and we notice that $I'$ is the ideal of a scheme of
$s'\leq s$ distinct points; notice that $s'>0$ as $X_1\not\in I$.
Clearly we have
\[I'+(X_1)\subseteq J=M^\perp:(X_1)+(X_1)=(X_1,X_2^{a_2+1},\ldots,X_n^{a_n+1}).\]
 Hence, for $t\gg 0$ we get
\[s'=HF(T/I',t)=\sum_{i=0}^t HF(T/(I'+(X_1)),i)\geq \]
\[\sum_{i=0}^t HF(T/J,i)=\frac 1{a_1+1} \Pi_{i=1}^n (a_i+1)\]
where the last equality holds as $J$ is a complete intersection
ideal. The conclusion then follows as $s\geq s'$.
\end{proof}

We now state and prove our main result.

\begin{thm}\label{mainthm}
Consider the degree $d$ form
\[F= M_1+ \ldots + M_r
\]
\[= x_{1,1} ^{a_{1,1}}\cdot\ldots\cdot x_{1,n_1}^{a_{1,n_1}}+
\ldots \ldots
+  x_{r,1}  ^{a_{r,1}}  \cdot\ldots\cdot x_{r,n_r}^{a_{r,n_r}}  ,  \]
where
\[a_{i,1}+ \ldots+{a_{i,n_i}}=d, \ \ \ \   1\leq a_{i,1}\leq\ldots \leq a_{i,n_i}, \ \ \     \   (1 \leq i \leq r ) .
\]

If $d=1$ then $\mathrm{rk}(F)=1$. If $d\geq 2$, then
$$\mathrm{rk}(F)=\sum _{i=1}^r
 \mathrm{rk}(M_i)
.$$

\end{thm}

\begin{proof}

The case $d=1$ is trivial as $F$ is a linear form, thus we only
have to prove the $d\geq 2$ case.
{
For $d=2$, $F$ is a quadratic form. Since  its  associated matrix  is {congruent} to a diagonal matrix of rank
$\sum _{i=1}^r n_i$ the conclusion follows. }

For $r=1$ the form $F$ is a monomial and the theorem is
proved  in  Proposition \ref{monmialsumofpoercorol}.

Hence we have
only to consider the cases $d>2$ and $r>1$.

By writing each monomial $M_i$ as a sum of powers we get a sum of
powers decomposition of $F$, thus we have $\mathrm{rk}(F)\leq \sum
_{i=1}^r\mathrm{rk}(M_i)$.
Hence, using Lemma
\ref{apolarityLEMMA}, it is enough to show that if $
F^\perp$ contains   the ideal of a scheme of $s$ distinct
points, then $s\geq \sum _{i=1}^r\mathrm{rk}(M_i)$.

Let $I \subseteq F^\perp$  be the ideal of a scheme  $\mathbb X $ of $s$ distinct
points, and let $ \mathbb X' \subseteq \mathbb X$ be the subsets of the $s'$ points of $\mathbb X$ not lying  on $\{X_{1,1}=\dots =  X_{r,1}=0$\}.
Let $I'$ be the  ideal of  $\mathbb X' $, i.e.,
\[  I'=I:(  X_{1,1},\dots ,  X_{r,1} )  .
\]
We will prove that $ s' \geq  \sum _{i=1}^r\mathrm{rk}(M_i)$, so the conclusion will
 follow as $s\geq s'$.

 The generic  linear derivation
 $\alpha_1 X_{1,1}+\ldots + \alpha_r X_{r,1}$ (where $\alpha_i \in  k $)  is not a zero divisor in $T/I'$.  Without loss of generality, and
possibly rescaling the variables, we may assume that
$$\partial= X_{1,1}+\ldots +  X_{r,1}$$
 is not a zero divisor in $T/I'$.
Hence,
for  $t\gg 0$ we get
\begin{equation} \label{0}
s' = HF(T/I',t)  = \sum_{i=0}^t HF(T/(I'+ (  \partial ) ),i) .
\end{equation}

Let $w$, ($0 \leq w \leq r$), be the number of $1$'s in the set
$\{ a_{1,1}, \dots, a_{r,1}
\}$. We may assume that
\[ a_{1,1}= \dots= a_{w,1} =1.
\]

We have
\[I' + (\partial)
\subseteq
 (F^\perp:( X_{1,1},\dots , X_{r,1} ))+ (\partial )
 \]
\[=  (F^\perp:( X_{1,1}) )\cap \ldots \cap  (F^\perp:( X_{r,1})) + (\partial )
\]
\[\subseteq
(x_{1,2} ^{a_{1,2}} \cdot\ldots\cdot x_{1,n_1}^{a_{1,n_1}})^\perp
\cap
\ldots \ldots
\cap ( x_{w,2}  ^{a_{w,2}}  \cdot\ldots\cdot
 x_{w,n_w}^{a_{w,n_w}} )^\perp
\]
\[
\cap
( ( x_{{w+1},1}  ^{a_{{w+1},1}-1}  \cdot
 x_{{w+1},2}  ^{a_{{w+1},2}}  \cdot \ldots\cdot
 x_{{w+1},n_{w+1}}^{a_{{w+1},n_{w+1}}} )^\perp + (X_{{w+1},1}) )
\]
\[
\cap
\ldots \ldots
\cap
( ( x_{r,1}  ^{a_{r,1}-1}  \cdot
 x_{r,2}  ^{a_{r,2}} \cdot \ldots\cdot
 x_{r,n_r}^{a_{r,n_r}} )^\perp + (  X_{r,1} ))
\]
\[
= J_1\cap   \ldots \cap  J_r, \]
where
\[J_{1}=(X_{1,1},X_{1,2}^{a_{1,2} +1}, \ldots,
X_{1,n_1}^{a_{1,n_1}+1},
X_{2,1},  \ldots, X_{2,n_2},
\ldots \ldots ,
X_{r,1},  \ldots, X_{r,n_r} );\]

\[J_{2}=(X_{1,1},\ldots,  X_{1,n_1},
X_{2,1}, X_{2,2} ^{a_{2,2}+1}\ldots, X_{2,n_2}^{a_{2,n_2}+1},
\ldots \ldots ,
X_{r,1},  \ldots, X_{r,n_r} );\]
\[ \vdots
\]
\[J_{r}=(X_{1,1},\ldots,  X_{1,n_1},
\ldots \ldots ,
X_{r,1}, X_{r,2} ^{a_{r,2}+1} \ldots, X_{r,n_r}  ^{a_{r,n_r}+1} ).\]

Observe that, if $n_i=1$, then $J_i$ is the maximal ideal.

So we have
\begin{equation} \label{1}
I' + (\partial) \subseteq  J_1\cap   \ldots \cap  J_r.
\end{equation}

The only linear forms in $ J_1\cap   \ldots \cap  J_r$ are
$X_{1,1},$ $ X_{2,1},$ $  \ldots , $ $X_{r,1}$, hence
\begin{equation} \label{2}
  \dim ( J_1\cap   \ldots \cap  J_r)_1 = r.
\end{equation}

Now we will prove by contradiction that in $I'$ there are no linear forms.
Assume that   $ L \in I' $ is a linear form,
$$L = \alpha _{1,1} X_{1,1}+\ldots+  \alpha _{1,n_1} X_{1,n_1}+
\cdots +
\alpha _{r,1} X_{r,1}+  \cdots+ \alpha _{r,n_r} X_{r,n_r}.
$$
Since $I' = I : (  X_{1,1},\dots ,  X_{r,1} ) $ we have
 $$X_{1,1}L,\dots ,  X_{r,1} L \in I  .$$
Hence $X_{i,1}L \in F^{\perp} $, ($1 \leq i \leq r$),  so
for $1 \leq i \leq w$ we get
$$ \alpha _{i,2}= \ldots = \alpha _ {i,n_i} =0,$$
and for $ i > w$ we get
$$ \alpha _{i,1}= \ldots = \alpha _ {i,n_i} =0.$$
Hence
$$L = \alpha _{1,1} X_{1,1}+\ldots+\alpha _{w,1} X_{w,1}.
$$

Let $ \mathbb X'' = \mathbb X \setminus \mathbb X'$, that is,  $\mathbb X''$ is the subsets of the points of $\mathbb X$  lying  on $\{X_{1,1}=\dots =  X_{r,1}=0$\}. Obviously
$ \{L=0 \}\supseteq \mathbb X'' $. It follows that $L \in I$.

 Since $I \subseteq F^\perp $, and in $ F^\perp$ there are no linear forms, we get a contradiction.

So we have
\begin{equation}  \label{3}
 \dim ( I' + (\partial) )_1 =1 . \end{equation}

Now by (\ref{0}), (\ref{1}),  (\ref{2}),  (\ref{3}), we get
\[
s' = \sum_{i=0}^t HF(T/(I'+ (  \partial ) ),i)
\]
\[=\dim T_1-1+ \sum_{i\neq1; \  i=0}^t HF(T/(I'+ (  \partial ) ),i)
\]
\[\geq \dim T_1-1+  \sum_{i\neq1; \ i=0}^t HF(T/ J_1\cap   \ldots \cap  J_r,i)
\]
\[=\dim T_1-1+  \sum_{ i=0}^t HF(T/ J_1\cap   \ldots \cap  J_r,i) -(\dim T_1 -r).
\]
Hence
\begin{equation} \label {7}
s'  \geq   \sum_{ i=0}^t HF(T/ J_1\cap   \ldots \cap  J_r,i) +r-1.
\end{equation}

We now need the following claim.

{\it Claim:}  For $t \gg 0$,
\[\sum_{i=0}^t HF(T/  J_1\cap \ldots \cap J_r ,i)=
\sum_{i=0}^t HF(T/  J_1 ,i)+ \ldots +\sum_{i=0}^t HF(T/  J_r ,i)-r+1
.
\]
\begin{proof}[Proof of the Claim:]
To prove the claim we proceed by induction on $r$. If $r=1$ the claim is obvious. Let $r>1$ and consider
 the following short exact sequence:
\[
0\longrightarrow T/ J_1\cap \ldots \cap J_r \longrightarrow
T/J_1\oplus T/ J_2\cap \ldots \cap J_r
\longrightarrow T/(J_1+ J_2\cap \ldots \cap J_r)\longrightarrow 0.
\]
By the inductive hypothesis, and since $J_1+ J_2\cap \ldots \cap J_r$ is the maximal ideal, we get the conclusion.
\end{proof}
Now we notice that  for $t \gg 0$  and since the $J_i$  are generated by regular sequences of length $n_1 + \dots +n_r$, we have
\[\sum_{i=0}^t HF(T/J_1,i)=  \frac 1{a_{1,1}+1}\Pi_{j=1}^{n_1 }  ( a_{1,j}+1) =
 \mathrm{rk}(M_1),
\]
\[ \vdots
\]
\[\sum_{i=0}^t HF(T/J_r,i)=  \frac 1{a_{r,1}+1}\Pi_{j=1}^{n_r }  ( a_{r,j}+1) = \mathrm{rk}(M_r).
\]
Hence by (\ref{7}) and the claim the conclusion immediately follows.

\end{proof}

\begin{rem}\label{leastvarrem}{
Let $F=\sum_1^r M_i$ be as in Theorem \ref{mainthm} and
$\mathbb{X}$ be a set of $s$ distinct points such that
$I_{\mathbb{X}}\subset F^\perp$. If $\mathbb X \cap
\{X_{1,1}=\dots = X_{r,1} =0\}=\mathbb{X}'\neq \emptyset$ is a set of $s'$ points, by the proof of the theorem we see that $s\geq s'+1 \geq\mathrm{rk}(F)+1$. In particular,
$\mathbb{X}$ does not have the least possible
cardinality if it intersects the special linear space
$\{X_{1,1}=\dots = X_{r,1} =0\}$.}
\end{rem}

\section{Applications}

We now present a few applications of our results.

\subsection{On the rank of the generic form}

It is well known, see \cite{AH95}, that for the generic degree $d$
form in $n$ variables $F$ one has
\[\mathrm{rk}(F)=\left\lceil{{d+n-1\choose d}\over n}\right\rceil.\]
However, the rank for a given specific form can be bigger or
smaller than that number. Moreover, it is trivial to see that
every form of degree $d$ is a sum of ${d+n\choose d}$ powers of
linear forms.  But, in general, it is not known how big the rank
of a degree $d$ form can be.

Using monomials we can try to produce explicit examples of forms
having rank bigger than that of the generic form. We give a
complete description of the situation for the case of three
variables. In that case, for $d \gg 0$, there are degree $d$
monomials with rank bigger than that the generic form, see
Proposition \ref{nleq3monomial}. However, for more than three
variables, this is no longer the case, see Remark
\ref{ngeq4monomial}.

\begin{prop}\label{nleq3monomial} Let $n=3$ and $d>2$ be an integer. Then
\[
\mbox{max}\left\{\mbox{rk}(M): M \in S_d\mbox{ is a
monomial}\right\}= \left\{\begin{array}{ll}\left({d+1\over
2}\right)^2 & d\mbox{ is odd} \\ \\ {d\over 2}\left({d\over
2}+1\right) & d\mbox{ is even}\end{array}\right.
\]
and this number is asymptotically ${3\over 2}$ of the rank of the
generic degree $d$ form in three variables, i.e.
\[\mbox{max}\left\{\mbox{rk}(M): M \in S_d\mbox{ is a
monomial}\right\}\simeq {3\over 2}\left\lceil{{d+2\choose 2}\over
3}\right\rceil\] for $d\gg 0$.
\end{prop}
\begin{proof}
We consider monomials $x_1^{a_1}x_2^{a_2}x_3^{a_3}$ with the
conditions $1\leq a_1\leq a_2\leq a_3$,
\[a_1+a_2+a_3=d\]
and we want to maximize the function $f(a_2,a_3)=(a_2+1)(a_3+1)$.
Considering $a_1$ as a parameter we are reduced to an optimization
problem in the plane where the constraint is given by a segment
and the target function is the branch of an hyperbola. For any
given $a_1$, it is easy to see that the maximum is achieved when
$a_2$ and $a_3$ are as close as possible to ${d-a_1 \over 2}$.
Also, when $a_1=1$ we get the maximal possible value. In
conclusion $\mathrm{rk}(M)$ is maximal for the monomial
\[ M=x_1x_2^{d-1\over 2}x_3^{d-1\over 2} (d \mbox{ odd}) \mbox{ or } M=x_1x_2^{{d\over 2} - 1}x_3^{{d\over 2}} (d\mbox{ even}).\]
With a straightforward computation, one easily sees that the rank
of the generic form is asymptotically ${d^2 \over 6}$, while the
maximal rank of a degree $d$ monomial is asymptotically ${d^2
\over 4}$. The conclusion follows.

\end{proof}

\begin{rem}\label{ngeq4monomial} { If $n\geq 4$ and $d\gg 0$, the degree
$d$ monomials do not provide examples of high rank forms. For
example, let $d=(n-1)k+1$ and consider a highest rank degree $d$
monomial
\[M=x_1x_2^k\cdot \ldots\cdot x_n^k.\]
If $F$ is a generic degree $d$ form, then
\[\mathrm{rk}(M)\simeq {n!\over (n-1)^{n-1} }\mathrm{rk}(F)\]
for $d\gg 0$ and we note that ${n!\over (n-1)^{n-1} }\leq 1$ if
$n\geq 4$. Hence, for each $n\geq 4$, there are infinitely many
values of $d$ for which no degree $d$ monomial has rank bigger
than the generic form.}
\end{rem}

\subsection{Sum of powers decomposition for polynomials}

Since we now know the rank of any given monomial, we can give a
description of one of its minimal sum of powers decompositions. An
explicit form for the scalars $\gamma$ can be found in
\cite{BBT2012} and it was also noticed by G. Whieldon
\cite{Whieldon}. In Remark \ref{sumofcoprimerem} we see how to use
this to obtain a minimal sum of powers decomposition for the sum
of coprime monomials.

\begin{prop}\label{sumofpowerdecmon}{  For integers  $1\leq a_1\leq \ldots \leq a_n$ consider the monomial
\[M=x_1^{a_1}\cdot\ldots\cdot x_n^{a_n}\]
and let $\mathcal{Z}(i)=\{z\in\mathbb{C} : z^{a_i+1}=1\}$. Then
\[M=\sum_{\epsilon(i)\in\mathcal{Z}(i),  i=2,\ldots,n}\gamma_{\epsilon(2),\ldots,\epsilon(n)} \left(x_1+\epsilon(2)x_2+\ldots +\epsilon(n)x_n\right)^d\]
where  the $\gamma_{\epsilon(2),\ldots,\epsilon(n)}$ are scalars
and this decomposition involves the least number of summands.}
\end{prop}
\begin{proof}
{ Another consequence of \cite[Lemma 1.15]{IaKa} allows one to
write a form as a sum of powers of linear forms. If $I\subset
M^\perp$ is an ideal of $s$ points, then
\[M=\sum_{j=1}^{s}\gamma_j \left(\alpha_j(1)x_1+\alpha_j(2)x_2+\ldots +\alpha_j(n)x_n\right)^d\]
where the $\gamma_j$ are scalars and $[\alpha_1:\ldots:\alpha_n]$
are the coordinates of the points having defining ideal $I$. Given
$M$ we can choose the following ideal of points
\[I=(y_2^{a_2+1} - y_1^{a_2+1}, y_3^{a_3+1} - y_1^{a_3+1}, \ldots , y_n^{a_n+1} - y_1^{a_n+1} )
.\]

It is straightforward to see that the points defined by $I$ have
coordinates
\[[1:\epsilon(2):\ldots:\epsilon(n)]\]
where $\epsilon(i)\in\mathcal{Z}(i)$. Renaming the scalars and
taking all possible combinations of the roots of $1$ we get the
desired $\mathrm{rk}(M)=\Pi_{i=2}^n(a_i+1)$ points and the result
follows.}

\end{proof}

\begin{rem}{In order to find an explicit decomposition for a given
monomial it is enough to solve a linear system of equations to
determine the $\gamma_j$. For example, in the very simple case of
$M=x_0x_1x_2$, we only deal with square roots of $1$ and we get:
\[x_0x_1x_2={1\over 24}(x_0+x_1+x_2)^3-{1\over 24}(x_0+x_1-x_2)^3-{1\over 24}(x_0-x_1+x_2)^3+{1\over
24}(x_0-x_1-x_2)^3.\]}
\end{rem}

\begin{rem}\label{sumofcoprimerem}{Using Proposition \ref{sumofpowerdecmon} we can easily
find a minimal sum of powers decomposition for the sum of coprime
monomials. If $F=M_1+\ldots +M_r$, then a minimal sum of powers
decomposition of $F$ is obtained by decomposing each $M_i$ as
described in Proposition \ref{sumofpowerdecmon}.}
\end{rem}

\begin{rem}{
Let $F=\sum_1^r M_i$ be the sum of coprime monomials, and
$F=\sum_1^{\mathrm{rk}(F)} L_i^d$ be a minimal sum of powers
decomposition of $F$. By Remark \ref{leastvarrem} we get that each
linear form $L_i$ must involve the variable
$X_{1,i}, i=1,\ldots,r$, where these are the variables with the
least exponent in each $M_i$. A particular instance of this
property, for $r=1$, has been noticed in \cite{BBT2012}}\end{rem}


\begin{thebibliography}{CCG11}

\bibitem[AH95]{AH95}
J.~Alexander and A.~Hirschowitz.
\newblock Polynomial interpolation in several variables.
\newblock {\em J. Algebraic Geom.}, 4(2):201--222, 1995.

\bibitem[BBT12]{BBT2012}
W.~Buczynska, J.~Buczynski, and Z.~Teitler.
\newblock Waring decompositions of monomials.
\newblock 2012.

\bibitem[CCG11]{CCG11}
E.~Carlini, M.V. Catalisano, and A.V. Geramita.
\newblock The solution to waring's problem for monomials.
\newblock 2011.

\bibitem[Ger96]{Ge}
A.V. Geramita.
\newblock Inverse systems of fat points: {W}aring's problem, secant varieties
  of {V}eronese varieties and parameter spaces for {G}orenstein ideals.
\newblock In {\em The Curves Seminar at Queen's, Vol.\ X (Kingston, ON, 1995)},
  volume 102 of {\em Queen's Papers in Pure and Appl. Math.}, pages 2--114.
  Queen's Univ., Kingston, ON, 1996.

\bibitem[IK99]{IaKa}
A.~Iarrobino and V.~Kanev.
\newblock {\em Power sums, {G}orenstein algebras, and determinantal loci},
  volume 1721 of {\em Lecture Notes in Mathematics}.
\newblock Springer-Verlag, Berlin, 1999.

\bibitem[LM04]{LM}
J.M. Landsberg and L.~Manivel.
\newblock On the ideals of secant varieties of {S}egre varieties.
\newblock {\em Found. Comput. Math.}, 4(4):397--422, 2004.

\bibitem[LT10]{LandsbergTeitler2010}
J.~M. Landsberg and Z.~Teitler.
\newblock On the ranks and border ranks of symmetric tensors.
\newblock {\em Found. Comput. Math.}, 10(3):339--366, 2010.

\bibitem[RS11]{RS2011}
K.~Ranestad and F.~Schreyer.
\newblock On the rank of a symmetric form.
\newblock 2011.


\bibitem[W11]{Whieldon}
G.~Whieldon
\newblock Private communication.
\newblock December 10 2011.


\end{thebibliography}
\end{document}